\documentclass[11pt,twoside]{article}
\usepackage{amsmath}
\usepackage{amsfonts}
\usepackage{amssymb}
\usepackage{amsthm}
\usepackage{amscd}
\usepackage[english]{babel}
\usepackage[latin1]{inputenc}
\newtheorem{theorem}{Theorem}[section]
\newtheorem{lemma}[theorem]{Lemma}
\newtheorem{proposition}[theorem]{Proposition}
\newtheorem{corollary}[theorem]{Corollary}
\theoremstyle{definition}
\newtheorem{definition}[theorem]{Definition}

\newtheorem{remark}[theorem]{Remark}

\newcommand{\g}{g_0}
\newcommand{\m}{{M}}
\newcommand{\mo}{M_0}
\newcommand{\h}{{\mathcal{H}}}

\newcommand{\R}{{\mathbb R}}

\begin{document}

\vbox{\hbox{\small International Meeting on Differential Geometry, C\'ordoba 2010}}
\hbox{\small November 15-17, University of C\'ordoba, Spain}
\vskip 1.5truecm


\begin{center}
{\large{\bf Remarks on the completeness  \\ of trajectories of
accelerated particles in \\  Riemannian manifolds and plane
waves\footnote{Partially supported by Spanish Grants with FEDER
funds MTM2010-18099 (MIC\-INN) and P09-FQM-4496 (J. de Andaluc\'\i
a).}}}
\end{center}

\medskip

\centerline{\bf A.M. Candela\footnote{Partially supported by
M.I.U.R. (research funds ex 40\% and 60\%) and of the G.N.A.M.P.A.
Research Project 2011 {\it Analisi Geometrica sulle Variet\`a di
Lorentz ed Applicazioni alla Relativit\`a Generale.}}, A. Romero
and M. S\'anchez}

\thispagestyle{empty}

\markboth{A.M. Candela, A. Romero, M. S\'anchez}{Completeness of Trajectories and Plane Waves}

\bigskip

\begin{abstract}
Recently, classical results on completeness of  trajectories of
Hamiltonian systems obtained at the beginning of the seventies,
have been revisited, improved and applied to Lorentzian Geometry
\cite{CRS2012}. Our aim here is threefold: to give explicit proofs
of some technicalities in the background of the specialists, to
show that the introduced tools allow to obtain more results for
the completeness of the trajectories, and to apply these results
to the completeness of spacetimes that generalize classical plane
and pp--waves.
\end{abstract}
\smallskip

\noindent
{\footnotesize {\sl MSC2010}: 34A26, 34C40, 37C60, 35Q75, 83C35.}

\vspace*{-2mm}

\section{Introduction}

In Classical Mechanics, one of the most venerable equations on a
(connected) Riemannian manifold $(M_0,g_0)$ is:
\[
\hspace*{3.5cm}\frac{D\dot\gamma}{dt}(t)\ =\ - \nabla^{M_0} V (\gamma(t),t)\hspace*{3.5cm}
(E_0)
\]
 where $D/dt$ denotes the covariant derivative along $\gamma$
induced by the Levi--Civita connection of $g_0$ and $\dot\gamma$
represents the velocity field along $\gamma$, while $V:M_0\times\R
\rightarrow \R$ is a smooth time--dependent potential. In fact,
when $(M_0,g_0)$ is $\R^3$, this is just Newton's second law for
forces that come from an external time-dependent potential. A
basic property that may have its solutions is  {\em completeness}
i.e. the extendability of their domain to all $\R$. At the
beginning of the seventies, some authors studied systematically
this property (see, e.g., \cite{Eb,Go,WM} or also \cite[Theorem
3.7.15]{AM}) but, essentially, they  focus only in the autonomous
case, that is, when $V(x,t)\equiv V(x)$ ($V$ is independent of
time).

Very recently, the authors have considered the completeness of the
trajectories not only for the general equation  (E$_0$) but also
for more general forces (see \cite{CRS2012}). Concretely,  $-
\nabla^{M_0} V$ was generalized to an arbitrary time-dependent
vector field $X$ and forces linearly dependent with the velocity
by means of an operator $F$, were also allowed. Nevertheless, it
is specially interesting to understand and analyze accurately the
differences between the autonomous and the non-autonomous case for
a potential. Moreover, as pointed out in \cite{CFS} (see also
\cite{CaSa}), the completeness for (E$_0$) is equivalent to the
completeness for the geodesics of a class of relativistic
spacetimes that generalizes the classical plane and pp--waves. So,
the aim of the present paper is, first, to analyze further  the
completeness in the non-autonomous case $X=- \nabla^{M_0} V$ (even
admitting the linear dependence of the force with the operator
$F$, see equation (E) below) and, then, to analyze  the
applications to generalized plane waves.

This paper is organized as follows. In Section 2, we recall the
framework for the completeness of Riemannian trajectories
(Subsection 2.1), and give a new theorem on completeness
(Subsection 2.2). The  proofs of two results are provided. The
first one is a technical comparison lemma that is commonly taken
into account in the results on completeness (Lemma
\ref{comparison}). The second one is a theorem on completeness
(Theorem \ref{G01}), obtained by developing further the techniques
in \cite{CRS2012}. In Section 3 we introduce plane wave type
spacetimes (Subsection 3.1) and explain the relation between the
problem of completeness of trajectories and the geodesic
completeness of generalized plane waves (Subsection 3.2).
Moreover, we give  further results on geodesic completeness
(Corollaries \ref{complete2}, \ref{complete3}) as a consequence of
the previous result of completeness of trajectories.

\section{Completeness of Riemannian trajectories}

\subsection{Framework}

Let $(\mo,g_0)$ be a (connected) smooth $n$--dimensional
Riemannian manifold and $V:\mo\times\R \rightarrow \R$ a given
smooth function. Taking $p\in \mo$ and $v\in T_p\mo$, there exists
a unique inextensible smooth curve $\gamma : I \to \mo$, $0\in I$,
solution of $(E_0)$ which satisfies the initial conditions
\begin{equation}\label{initial}
\gamma(0) = p,\quad \dot\gamma(0) = v.
\end{equation}
An inextensible solution of $(E_0)$ is {\sl complete} if it is
defined on the whole real line. Note that equation $(E_0)$ in the
trivial case $V\equiv 0$ is the equation of the geodesics in
$(\mo,g_0)$. Let us recall that a Riemannian manifold $(\mo,g_0)$
is {\sl geodesically complete} if any of its inextensible
geodesics is defined on $\mathbb{R}$ or, equivalently, the metric
distance induced by $g_0$ is complete.

In \cite[Theorem 2.1]{Go} Gordon proved the completeness of the
trajectories of $(E_0)$ if the potential $V$ is time--independent,
bounded from below and satisfying either $(\mo,g_0)$ is complete
or $V$ is proper (i.e., $V^{-1}(K)$ is compact in $\mo$ for any
compact $K\subset \R$). Other results in the autonomous case were
given in \cite{Eb,WM} and \cite[Theorem 3.7.15]{AM}.

Following \cite{CRS2012}, we generalize such results to the
non--autonomous case by including also the action of a (1,1)
tensor field $F$ along the natural projection $\pi : \mo \times \R
\longrightarrow \mo$, i.e., we consider the second order
differential equation
\[
\hspace*{2.9cm}\frac{D\dot\gamma}{dt}(t)\ =\ F_{(\gamma(t),t)}\
\dot\gamma(t) - \nabla^{M_0} V (\gamma(t),t).\hspace*{2cm}(E)
\]
Let us remark that the existence and uniqueness result of
inextensible solutions of $(E)$, under the same initial conditions
\eqref{initial}, remains now true, and, obviously, one has the notion of
complete inextensible trajectory of $(E)$.

Now, let us introduce some terminology in order to express natural
conditions on $F$ and $V$. Notice  that, in general, $F$ is
neither self-adjoint nor skew-adjoint with respect to $g_0$, and
denote by $S$ the self--adjoint part of $F$.
For each $t \in \R$, put
\[
\| S(t) \|\ : = \
\max\big\{\big|S_{\sup}(t)\big|,\, \big|S_{\inf}(t)\big|\big\}
\]
where
\[
S_{\sup}(t): = \sup_{\underset{\|v\|=1}{v\in T\mo}}
g\left(v,S_{(p,t)} v\right) \quad \text{and} \quad S_{\inf}(t):=
\inf_{\underset{\|v\|=1}{v\in T\mo}} g\left(v,S_{(p,t)} v\right).
\]
We say that $S$ is {\em bounded} (resp. {\em upper bounded}, {\em
lower bounded}) {\em along finite times} when, for each $T>0$,
there exists a constant $N_T$ such that
\begin{equation}\label{bf}
\|S(t)\|  \le N_T \; \hbox{(resp. $S_{\sup}(t) \le N_T$, $-S_{\inf}(t) \le N_T$)}\; \hbox{for all $t\in [-T,T]$.}
\end{equation}
Moreover, the potential $V$ is {\em bounded from below along
finite times} if there exists a continuous function $\beta_0:
\R\rightarrow \R$ such that
\begin{equation}\label{bf1}
V(p,t)\ \ge\ \beta_0(t)\quad \hbox{for all}\quad (p,t)\in \mo
\times \R.
\end{equation}

In order to investigate the completeness of the inextensible
solutions of equation $(E)$, let us recall that an integral curve
$\rho$ of a vector field on a manifold, defined on some bounded
interval $[a,b)$, $b<+\infty$, can be extended to $b$ (as an
integral curve) if and only if there exists a sequence
$\{t_n\}_n$, $t_n \to b^-$, such that $\{\rho(t_n)\}_n$ converges
\cite[Lemma 1.56]{ON}. The following technical result follows
directly from this fact and \cite[Lemma 3.1]{CRS2012}.

\begin{lemma}\label{extend}
Let $\gamma: [0,b) \to \mo$ be a solution of equation $(E)$ with
$0<b<+\infty$. The curve $\gamma$ can be extended to $b$ as a
solution of $(E)$ if and only if there exists a sequence
$\{t_n\}_n \subset [0,b)$ such that $t_n \to b^-$ and the sequence
of velocities $\{\dot\gamma(t_n)\}_n$ is convergent in the tangent
bundle $T\mo$.
\end{lemma}

Furthermore, we need also the following result (compare with
\cite[Example 2.2.H]{AM}).

\begin{lemma}[{\bf Comparison Lemma}]\label{comparison}
Let $\varphi :[a,+\infty) \to \R$ be a continuous
monotone increasing function such that
\begin{equation}\label{et1}
\varphi(s) > 0\quad\ \hbox{for all $\ s\ge a$}
\quad\hbox{and}\quad \int_a^{+\infty} \frac{ds}{\varphi(s)} =
+\infty.
\end{equation}
If a $C^1$ function $v_0=v_0(t)$ satisfies the equation
\begin{equation}\label{global}
v_0'(t)\ =\ \varphi(v_0(t))\quad \hbox{with $v_0(0) \ge a$,}
\end{equation}
and it is inextensible, then it is defined for all $t \ge 0$.

\vspace{1mm}

Furthermore, if $v :[0,b) \to \R$ is a continuous function such
that
\begin{equation}\label{stime}
\left\{\begin{array}{ll} \displaystyle a \le v(t) \le v(0) + \int_0^t
\varphi(v(s)) d s  &\hbox{for all \;
$t \in [0,b)$,}\\[1mm]
v(0) \le v_0(0),&
\end{array}\right.
\end{equation}
then $v(t) \le v_0(t)$ for all $t \in [0,b)$.
\end{lemma}

\begin{proof} Even though this is a simple exercise,  we
prefer to give here a complete argument by completeness. If
$v_0=v_0(t)$ is a $C^1$ inextensible solution of \eqref{global} in
the interval $[0,\bar b)$, then
\begin{equation}\label{et2}
v_0(t) \ge v_0(0) \ge a\quad \hbox{for all $t\in [0,\bar b)$,}
\end{equation}
whence, for all $t\in [0,\bar b)$, $\varphi(v_0(t))(>0)$ is well
defined and $v_0$ becomes strictly monotone increasing. Thus,
dividing both the terms of \eqref{global} by $\varphi(v_0(t))$ and
integrating in $[0,t]$, $0 < t < \bar b$, we have
\[
\int_0^t\frac{v_0'(\tau)}{\varphi(v_0(\tau))}d\tau = t,
\]
hence, $v_0 = v_0(t)$ is the inverse of
\begin{equation}\label{et}
t(v_0)\ =\ \int_{v_0(0)}^{v_0} \frac{ds}{\varphi(s)},
\end{equation}
with the maximum $\bar b$ equal to
$\displaystyle\lim_{v_0\rightarrow +\infty}t(v_0)$ in \eqref{et}.
From \eqref{et1} it follows $\bar b = +\infty$.

\vspace{1mm}

Now, let $v=v(t)$, $t\in [0,b)$, be such to satisfy \eqref{stime}
and define
\[
h(t) \ =\ v_0(0) +\ \int_0^t \varphi(v(s)) ds.
\]
Clearly, $h$ is a $C^1$ function such that
\[
h(0) = v_0(0)\quad\hbox{and}\quad h'(t) = \varphi(v(t))
\quad \hbox{for all $t\in [0,b)$.}
\]
Moreover, from \eqref{stime} it follows
\begin{equation}\label{stima3}
a\ \le\ v(t) \ \le\ h(t)\quad \hbox{for all $t\in [0,b)$,}
\end{equation}
whence the monotonicity of $\varphi$ implies
\begin{equation}\label{stima2}
h'(t)\ \le\ \varphi(h(t))\quad \hbox{for all $t\in [0,b)$.}
\end{equation}
Thus, from \eqref{et1}, \eqref{global} and \eqref{stima2} we have
\[
\frac{h'(t)}{\varphi(h(t))} \ \le\ 1\ =\
\frac{v_0'(t)}{\varphi(v_0(t))} \quad \hbox{for all $t\in [0,b)$,}
\]
whence direct computations give
\begin{equation}\label{et4}
\int_{v_0(0)}^{h(t)}\frac{ds}{\varphi(s)} \ \le\
\int_{v_0(0)}^{v_0(t)}\frac{ds}{\varphi(s)}
\quad \hbox{for all $t\in [0,b)$.}
\end{equation}
Now, assume that
$\bar t \in (0,b)$ exists such that $h(\bar t) > v_0(\bar t)$.
Hence, \eqref{et1} and \eqref{et2} imply
\[
\int_{v_0(0)}^{h(\bar t)}\frac{ds}{\varphi(s)} \ >\
\int_{v_0(0)}^{v_0(\bar t)}\frac{ds}{\varphi(s)}
\]
in contradiction with \eqref{et4}. So, we have
$h(t) \le v_0(t)$ for all $t \in [0,b)$ and the proof follows from \eqref{stima3}.
\end{proof}

\subsection{Our main result on the non--autonomous problem $(E)$}

Now, we are ready to state our main result on the completeness of inextensible trajectories
of the non--autonomous problem $(E)$.

\begin{theorem}\label{G01} Let $(\mo,\g)$ be a complete Riemannian mani\-fold, $F$ a
smooth time--dependent $(1,1)$ tensor field with self--adjoint
component $S$ and $V:\mo\times\R \rightarrow \R$ a smooth
potential. Assume that $\|S(t)\|$ is bounded along finite times,
$V$ is bounded from below along finite times
and there exists a continuous function $\alpha_0: \R\rightarrow \R$
such that
\[
\left|\frac{\partial V}{\partial t}(p,t)\right|\ \le\ \alpha_0(t)
(V(p,t) - \beta_0(t))\quad \hbox{for all \, $(p,t)\in \mo\times \R$}
\]
with $\beta_0$ as in \eqref{bf1}.

Then, each inextensible solution of equation $(E)$ must be complete.
\end{theorem}

The proof of Theorem \ref{G01} is a direct consequence
of the following more general result.

\begin{proposition}\label{G011} Let $(\mo,\g)$ be a complete Riemannian manifold, $F$ a
smooth time--dependent $(1,1)$ tensor field with self--adjoint
component $S$ and $V:\mo\times\R \rightarrow \R$ a smooth
potential bounded from below along finite times with $\beta_0$ as in \eqref{bf1}.

\vspace{1mm}

If $S_{\sup}(t)$ is upper bounded along finite times
and a continuous function $\alpha_0: \R\rightarrow \R$ exists
such that
\[
\frac{\partial V}{\partial t}(p,t)\ \le\ \alpha_0(t)
(V(p,t) - \beta_0(t))\quad \hbox{for all \, $(p,t)\in \mo\times
\R$,}
\]
then each inextensible solution of equation $(E)$ must be forward
complete.

\vspace{1mm}

Conversely, if $S_{\inf}(t)$ is lower bounded along finite times
and a continuous function $\alpha_0: \R\rightarrow \R$ exists such
that
\[
- \frac{\partial V}{\partial t}(p,t)\ \le\ \alpha_0(t)
(V(p,t) - \beta_0(t))\quad \hbox{for all \, $(p,t)\in \mo\times
\R$,}
\]
then each inextensible solution of equation $(E)$ must be backward
complete.
\end{proposition}

\begin{proof}
Let $\gamma$ be a non--constant forward inextensible solution of
equation $(E)$ defined on the interval $[0,b) \subset \R$. Arguing
by contradiction, assume that $\gamma$ is not forward complete,
i.e., $b<+\infty$, so a real positive constant $T > b$ can be
fixed so that \eqref{bf} holds for $S_{\sup}(t)$, furthermore
\begin{equation}\label{stima5}
V(p,t) - B_T \ge 1\quad \hbox{and}\quad
\frac{\partial V}{\partial t}(p,t)\ \le\ A_T
(V(p,t) - B_T)
\end{equation}
for all \, $(p,t)\in \mo\times [-T,T]$, with $A_T \ge \max \alpha_0([-T,T])$
and $B_T \le \min\beta_0([-T,T]) -1$.

Now, for simplicity, denote
\[
u(t) = g(\dot\gamma(t),\dot\gamma(t))\;\; \hbox{and}\;\;
v(t)\ =\ \frac12\ u(t) + V(\gamma(t),t) - B_T,
\quad t\in [0,b).
\]
From \eqref{stima5} it follows
\[
u(t)+1\ \le\ 2 v(t),
\]
hence if $v(t)$ is bounded in $[0,b)$ so is $u(t)$, that is
a constant $k > 0$ exists such that
\begin{equation}\label{inequality3}
u(t)\ \le\ k \quad\hbox{for all $t \in [0,b)$.}
\end{equation}
Note that this inequality is enough for contradicting that $b$ is
finite. In fact, \eqref{inequality3} implies that
$\dot\gamma([0,b))$  is bounded in $T\mo$ and, being $(\mo,\g)$
complete, Lemma \ref{extend} is applicable because of the
completeness of $M_0$. Hence, $\gamma$ can be extended to $b$ in
contradiction with its maximality assumption.

In order to prove that $v(t)$ is bounded in $[0,b)$,
taking any $t\in [0,b)$ by using equation $(E)$
and estimates \eqref{bf} and \eqref{stima5} we have
\[\begin{split}
\frac{d v}{dt}(t)\ &=\
g\big(\frac{D\gamma}{dt}(t),\dot\gamma(t)\big)
+ g\left(\nabla^{\mo}V(\gamma(t),t),\dot\gamma(t)\right) + \frac{\partial V}{\partial t}(\gamma(t),t)\\
&=\ g\big(F_{(\gamma(t),t)}\dot \gamma(t),\dot\gamma(t)\big)\, + \,\frac{\partial V}{\partial
t}(\gamma(t),t)\\[1mm]
&=\ g\big(S_{(\gamma(t),t)}\dot \gamma(t),\dot\gamma(t)\big)\, + \,\frac{\partial V}{\partial
t}(\gamma(t),t)\\[1mm]
&\le\ N_T\, u(t) + A_T \big(V(\gamma(t),t)-B_T\big).
\end{split}
\]
Whence, $A_T^* \in \R$ exists such that
\begin{equation}\label{G05}
\frac{d v}{dt}(t)\ \le\ A_T^*\, v(t) \quad\hbox{for all $t \in
[0,b)$.}
\end{equation}
On the other hand, if we consider the linear equation
\begin{equation}\label{G06}
w'(t)\ =\ A_T^*\ w(t),
\end{equation}
let $v_0=v_0(t)$ be the unique (global) solution of \eqref{G06}
satisfying the initial condition $v_0(0) = v(0)$, with $v(0) \ge 1$ from \eqref{stima5}.
Thus, from \eqref{G05} and Lemma \ref{comparison} with $\varphi(s) =  A_T^* s$
and $a=1$, we have that $v(t) \le v_0(t)$ for all $t\in [0,b)$, with $v_0(t)$
bounded in $[0,b]$; whence, $v(t)$ is bounded in $[0,b)$.

\vspace{1mm}

Conversely, let us assume that $\gamma$ is not backward complete
in $(-b,0]$ with $b < +\infty$, then we can consider $T>b$ and
$\tilde \gamma(t) := \gamma(-t)$ in $[0,b)$. From the lower
boundedness of $S_{\inf}(t)$ in $[-T,T]$ and the estimate on
$-\frac{\partial V}{\partial t}$ along finite times, we have
\[
\begin{split}
\frac{dv}{dt}(-t)\ &=\
- g\left(S_{(\gamma(-t),-t)} \dot\gamma(-t),\dot\gamma(-t)\right)
- \frac{\partial V}{\partial t}(\gamma(-t),-t)\\[1mm]
&\le\ N_T\, u(-t) + A_T \big(V(\gamma(-t),-t) - B_T\big),
\end{split}
\]
and we repeat the above argument for $\tilde\gamma(t)$.
\end{proof}

\begin{remark}
Both in Theorem \ref{G01} and in Proposition \ref{G011} the
assumption on the completeness of $(\mo,\g)$ can be replaced by
the condition ``{\sl $V$ is proper}''. In fact, in the above proof
once we have proven that $v(t)$ is bounded in $[0,b)$, the
properness of $V$ implies that $\dot \gamma ([0,b))$ lies in a
compact subset of $T\mo$, so $\gamma$ can be extended to $b$.
\end{remark}

\begin{remark}
As commented in the Introduction, other completeness results on
the inextensible trajectories of equation $(E)$ as well as their
comparison with Theorem \ref{G01} can be found in \cite{CRS2012}.
\end{remark}

\section{Geodesic completeness of GPW}

\subsection{Plane waves and their generalizations}

A {\em parallely propagated wave} spacetime, or a {\em pp--wave} in
brief, is a relativistic spacetime $(\R^4,ds^2)$ where the
Lorentzian metric $ds^2$ has the form
\[
ds^2 \ =\ dx^2+dy^2 + 2dudv + H(x,y,u) du^2,
\]
being $(x,y,u,v)$ the natural coordinates of $\R^4$ and $H : \R^3
\to \R$ a non--zero smooth function. If the expression of $H$ is
quadratic in $x, y$, i.e.,
\begin{equation} \label{ehpw}
H(x,y,u)\ =\ f_1(u) x^2 - f_2(u) y^2 + 2 f(u) xy,
\end{equation}
for some smooth real functions $f_1$, $f_2$ and $f$, then the
spacetime is called {\em  plane wave}, and, in particular, an {\sl
(exact plane fronted) gravitational wave} if $f_1 \equiv f_2$ (for example, see \cite{BEE}).

Since the pioneer papers dealing with gravitational waves
\cite{Br,ER}, these spacetimes have been widely studied by many
authors (see \cite{CFS} and references therein or the summary in
\cite{Yu}) not only for their geometric interest but above all for
their physical interpretation. In fact, as explained in
\cite{MTW}, a gravitational wave represents ripples in the shape
of spacetime which propagate across spacetime, as water waves are
small ripples in the shape of the ocean's surface propagating
across the ocean. The source of a gravitational wave is the motion
of massive particles; in order to be detectable, very massive
objects under violent dynamics must be involved (binary stars,
supernovas, gravitational collapses of stars...). With more
generality, pp--waves may also taken into account the propagation
of non--gravitational effects such as electromagnetism.

Here,  we focus only on the  property of  geodesic
completeness. In particular, we add further information to
the study of the geometric properties for the family of
generalized plane waves, already developed in \cite{CFS,
CRS2012, FS_CQG, FS_JHEP}. The key fact is that the geodesic completeness
of a pp--wave reduces to the completeness of the inextensible
trajectories that are solutions of the second order differential
equation (E$_0$) when $(M_0,g_0)$ is $\R^2$. However, this last
restriction is not important and, following \cite{FS_CQG}, the
classical notion of pp--wave can be generalized as follows:

\begin{definition}\label{pfwave}
{\rm A Lorentzian manifold $(\m,g)$ is called \emph{generalized
plane wave}, briefly \emph{GPW}, if there exists a connected
$n$--dimensional Riemannian manifold $(\mo,\g)$ such that $\m =
\mo \times \R^{2}$ and
\begin{equation}\label{wave}
g\ =\ \g + 2dudv+ \h (x,u) du^{2},
\end{equation}
where $x\in \mo$, the variables $(u,v)$ are the natural
coordinates of $\R^{2}$ and the smooth function $\h: \mo\times
\R\rightarrow \R$ is such that $\h\not\equiv 0 $.}
\end{definition}

\subsection{Application to geodesic completeness}

In order to investigate the properties of geodesics in a GPW, it
is enough studying the behavior of the Riemannian trajectories
under a suitable potential $V$. In particular, the problem of
geodesic completeness is fully reduced to a purely Riemannian
problem: the completeness of the inextensible trajectories of
particles moving under the potential $V(x,u) = - \frac{1}2\,
\h(x,u)$ as the following result shows (see \cite[Theorem
3.2]{CFS} for more details).

\begin{theorem}
\label{traj} A GPW is geodesically complete if and only if
$(\mo,\g) $ is a complete Riemannian manifold and the inextensible
trajectories of
\[
\hspace*{45mm}\frac{D\dot{\gamma}}{dt}\ =\
\frac{1}{2}\,\nabla^{\mo}\h(\gamma(t),t) \hspace*{35mm} (E_0^*)
\]
are complete.
\end{theorem}

Now, we can use Theorem \ref{G01} to obtain the completeness of
the inextensible trajectories of equation $(E_0^*)$. Then, the
following result on the geodesic completeness on GPW can be
stated:

\begin{corollary}
\label{complete2} Let $\m=\mo\times\R^{2}$ be a GPW such that
$(\mo,\g)$ is a geodesically complete Riemannian manifold and
$\h:\mo\times \R\rightarrow \R$ is a smooth function. If there
exist two continuous functions $\alpha_0$, $\beta_0: \R\rightarrow
\R$ such that
\[
\h(x,u) \ \le\ \beta_0(u)\quad \hbox{and} \quad
\left|\frac{\partial \h}{\partial u}(x,u)\right|\ \le\ \alpha_0(u)
\left(\beta_0(u) - \h(x,u)\right)
\]
for all $(x,u) \in \mo \times\R$, then $(\m,g)$ is geodesically
complete.
\end{corollary}

We emphasize that other results on autonomous and non-autonomous
potentials can be translated into results of geodesic completeness
of GPW. So, as a consequence of \cite[Corollary 3.6]{CRS2012} we
have:
\begin{corollary}
\label{complete3} A GPW with complete $(\mo,\g)$  is geodesically
complete if $\nabla^M\h$ grows at most linearly in $M$ along
finite times
\end{corollary}

\begin{remark}
 The particular case of this corollary for pp--waves (i.e. its
application for $(\mo,\g)=\R^2$) was discussed in \cite{CRS2012},
and it has a clear interpretation: not only classical plane waves
are geodesically complete  but also each pp--wave such that its
coefficient $\h$ behaves qualitatively as the one of a plane wave,
are. This can be understood as a result of stability of the
completeness of plane waves in the class of all pp--waves. So,
Corollary \ref{complete3} also ensures stability of completeness
in the class of generalized plane waves.

Even though the physical interpretation of Corollary
\ref{complete2} is not so clear, it is logically independent of
Corollary \ref{complete3} (a discussion as the one below
Proposition 3.7 in \cite{CRS2012} also holds here). This shows
that the application of the techniques are not exhausted and,
under motivated assumptions,  further results could be obtained.
\end{remark}

\vspace{12mm}

\noindent
Anna Maria Candela\\
Dipartimento di Matematica, \\ Universit\`a degli Studi di Bari ``Aldo Moro'',\\
Via E. Orabona 4, 70125 Bari, Italy\\
e-mail: candela@dm.uniba.it

\vspace{9mm}

\noindent
Alfonso Romero$^\dagger$ and Miguel S\'anchez$^\ddagger$\\
Departamento de Geometr\'{\i}a y Topolog\'{\i}a,\\
Facultad de Ciencias, 
Universidad de Granada,\\
Avenida Fuentenueva s/n,\\
18071 Granada, Spain\\
e-mail addresses: $^\dagger$aromero@ugr.es,
$^\ddagger$sanchezm@ugr.es

\end{document}